\documentclass[11pt,a4paper]{article}
\usepackage[top=27mm, bottom=27mm, left=22mm, right=22mm]{geometry}
\usepackage{amsthm,amsmath,amssymb}
\usepackage{mathrsfs}
\usepackage{multirow}
\usepackage{amsthm}
\usepackage{microtype}
\usepackage{hyperref}
\usepackage{mathtools}
\usepackage{thm-restate}
\usepackage{thmtools}
\usepackage{tabularx}
\usepackage[table]{xcolor}
\usepackage[capitalize]{cleveref}
\newtheorem{theorem}{Theorem}[]
\newtheorem{lemma}[theorem]{Lemma}

\newtheorem{conjecture}[theorem]{Conjecture}

\newcommand{\ma}{\mathcal}

\newcommand{\mr}{\mathscr}
\newcommand{\s}{\subseteq}

\newcommand{\fr}{\frac}
\newcommand{\lc}{\lceil}
\newcommand{\rc}{\rceil}

\newcommand{\e}{{\rm{ex}}}

\begin{document}

\title{Degenerate Tur\'an densities of sparse hypergraphs II:\\ a solution to the Brown-Erd\H{o}s-S\'os problem for every uniformity}

\author{Chong Shangguan\footnote{Research Center for Mathematics and Interdisciplinary Sciences, Shandong University, Qingdao 266237, China. Email: theoreming@163.com}
}

\date{}
\maketitle

\begin{abstract}
\noindent For fixed integers $r\ge 3, e\ge 3$, and $v\ge r+1$, let $f_r(n,v,e)$ denote the maximum number of edges in an $n$-vertex $r$-uniform hypergraph in which the union of arbitrary $e$ distinct edges contains at least $v+1$ vertices. In 1973, Brown, Erd\H{o}s and S\'os proved that $f_r(n,er-(e-1)k,e)=\Theta(n^k)$ and
conjectured that the limit $\lim_{n\rightarrow\infty}\frac{f_3(n,e+2,e)}{n^2}$ always exists for all fixed integers $e\ge 3$. In 2020 Shangguan and Tamo conjectured that the limit $\lim_{n\rightarrow\infty}\frac{f_r(n,er-(e-1)k,e)}{n^k}$ always exists for all fixed integers $r>k\ge 2$ and $e\ge 3$, which contains the BES conjecture as a special case for $r=3, k=2$. Recently, based on a result of Glock, Joos, Kim, K\"uhn, Lichev, and Pikhurko, Delcourt and Postle proved the BES conjecture. Extending their result, we show that the limit $\lim_{n\rightarrow\infty}\frac{f_r(n,er-2(e-1),e)}{n^2}$ always exists, thereby proving the BES conjecture for every uniformity.
\end{abstract}

\section{Introduction}
We will focus on the sparse hypergraph problem introduced by Brown, Erd\H{o}s and S\'os \cite{Brown-Erdos-Sos-1} in 1973. To move forward, we will begin with some notations.

For an integer $r\ge 2$, an $r$-uniform hypergraph $\ma{H}$ (or $r$-graph, for short) on the vertex set $V(\ma{H})$, is a family of $r$-element subsets of $V(\ma{H})$,  called the {\it edges} of $\ma{H}$.
For fixed integers $r\ge 3, e\ge 3$, and $v\ge r+1$, an $r$-graph $\ma{H}$ is called {\it $\mr{G}_r(v,e)$-free} if the union of arbitrary $e$ distinct edges of it contains at least $v+1$ vertices, namely, for arbitrary distinct $A_1,\ldots,A_e\in\ma{H}$ one has $|\cup_{i=1}^e A_i|\ge v+1$. Due to the sparsity of its edges, such $r$-graphs are also termed {\it sparse hypergraphs} \cite{Furedi-Ruszinko-no-grid}. Let $f_r(n,v,e)$ denote the maximum number of edges in an $n$-vertex $\mr{G}_r(v,e)$-free $r$-graph.

In 1973, Brown, Erd\H{o}s and S\'os \cite{Brown-Erdos-Sos-1} initiated the study of the function $f_r(n,v,e)$, which has attracted considerable attention throughout the years.
More concretely, they showed that
\begin{align}\label{eq:BESbound}
   \Omega(n^{\fr{er-v}{e-1}})=f_r(n,v,e)=O(n^{\lc\fr{er-v}{e-1}\rc}).
  \end{align}

\noindent Improvements on lower or upper bounds of \eqref{eq:BESbound} for less general parameters were obtained in a series of works, see, e.g. \cite{Alon-Shapira-sparse,Brown-Erdos-Sos-1,conlon2019new,Erdos-Frankl-Rodl-sparse,Ge-Shangguan-spa-hyp-bou-con,Glock-triple,Rodl,Ruzsa-Szemeredi-63,Sarkozy-Selkow-extension-BES,Sarkozy-Selkow-extension-RS,Shangguan-Tamo-degenerate,Shangguan-Tamo-spa-hyp-coding,Sidorenko-approximate-steiner,Brown-Erdos-Sos-2}.

In this paper we are interested in the special case where $k=\frac{er-v}{e-1}\in\{1,\ldots,r-1\}$ is an integer. For $k=1$ the authors of \cite{Glock64} (see also \cite{erdosr=2k=1} (4) for the case $r=2$) noted that
\begin{align}\label{eq:k=1}
   \lim_{n\rightarrow\infty}\frac{f_r(n,er-(e-1),e)}{n}=\frac{(e-1)n}{(e-1)(r-1)+1}.
\end{align}
According to \eqref{eq:BESbound} it is easy to see that
\begin{align} \label{eq:exp}
    f_r(n,er-(e-1)k,e)=\Theta(n^k),
\end{align}

\noindent and it is natural to ask whether the limit
$$\lim_{n\rightarrow\infty}\frac{f_r(n,er-(e-1)k,e)}{n^k}$$
exists. For $e=2$ the above question is completely resolved by a classical result of R{\"o}dl \cite{Rodl-nibble}. So we will focus on the case $e\ge 3$. For $e\ge 3, r=3, k=2$, Brown, Erd\H{o}s and S\'os \cite{Brown-Erdos-Sos-1} posed the following conjecture:
\begin{conjecture}[see Section 5 in \cite{Brown-Erdos-Sos-1}]\label{con:besconjecture}
The limit
$\lim_{n\rightarrow\infty}\fr{f_3(n,e+2,e)}{n^2}$
exists for every fixed $e\ge 3$.
\end{conjecture}

In 2019 Glock \cite{Glock-triple} solved the $e=3$ case of \cref{con:besconjecture} by showing that $\lim_{n\rightarrow\infty}\fr{f_3(n,5,3)}{n^2}=\fr{1}{5}$. In a recent work, Glock, Joos, Kim, K\"uhn, Lichev, and Pikhurko \cite{Glock64} solved the $e=4$ case of \cref{con:besconjecture} by showing that $\lim_{n\rightarrow\infty}\fr{f_3(n,6,4)}{n^2}=\fr{7}{36}$. Quite recently, based on a result of \cite{Glock64} Delcourt and Postle \cite{bes} solved \cref{con:besconjecture} in the full generality, without determining the limit precisely.

In the spirit of \eqref{eq:exp} and Conjecture \ref{con:besconjecture} Shangguan and Tamo \cite{Shangguan-Tamo-degenerate} posed the following conjecture, which largely generalizes \cref{con:besconjecture}:
\begin{conjecture}[see Question 2 in \cite{Shangguan-Tamo-degenerate}]\label{que:ST-1}
The limit $\lim_{n\rightarrow\infty}\frac{f_r(n,er-(e-1)k,e)}{n^k}$  exists for all fixed integers $r>k\ge 2,e\ge 3$.
\end{conjecture}
\noindent In \cite{Shangguan-Tamo-degenerate}, such a limit, if it exists, is called
{\it the degenerate Tur\'an density of sparse hypergraphs}. Extending Glock's result from $r=3,k=2,e=3$ to arbitrary fixed $r\ge 4$, Shangguan and Tamo \cite{Shangguan-Tamo-degenerate} showed that
\begin{align*}
  \lim_{n\rightarrow\infty}\fr{f_r(n,3r-4,3)}{n^2}=\fr{1}{r^2-r-1}.
\end{align*}
Following this line of work, Glock, Joos, Kim, K\"uhn, Lichev, and Pikhurko \cite{Glock64} proved that
\begin{align}\label{eq:exact}
  \lim_{n\rightarrow\infty}\fr{f_r(n,3r-2k,3)}{n^k}=\fr{2}{k!(2\binom{r}{k}-1)}\quad\text{and}\quad\lim_{n\rightarrow\infty}\fr{f_r(n,4r-3k,4)}{n^k}=\fr{1}{k!}\binom{r}{k}^{-1},
\end{align}
thereby settling \cref{que:ST-1} completely for $e\in\{3,4\}$.

The main result of this paper is the following theorem, which settles \cref{que:ST-1} for $k=2$ and proves \cref{con:besconjecture} for every uniformity.

\begin{theorem}\label{thm:main}
The limit $\lim_{n\rightarrow\infty}\frac{f_r(n,er-2(e-1),e)}{n^2}$  exists for all fixed integers $r\ge 3$ and $e\ge 3$.
\end{theorem}

The idea of our proof is sketched as follows. Let us begin with an attempt that aims to prove the general case of \cref{que:ST-1}. We will make use of the following result proved in \cite{Glock64}.

\begin{lemma}[see Corollary 3.2 in \cite{Glock64}]\label{lem:exist-0}
    Let $\ma{F}$ be an $m$-vertex $r$-graph which simultaneously $\mr{G}_r(er-(e-1)k,e)$-free and $\mr{G}_r(ir-(i-1)k-1,i)$-free for all $2\le i\le e-1$. Then  $$\liminf_{n\rightarrow\infty}\frac{f_r(n,er-(e-1)k,e)}{n^k}\ge\fr{|\ma{F}|}{m^k}.$$
\end{lemma}

For $t\in\{2,\ldots,e-1\}$, let $f_r^{(t)}(n,er-(e-1)k,e)$ denote the maximum number of edges in an $n$-vertex $r$-graph which simultaneously satisfy all of the following properties:
\begin{align}\label{eq:property}
    \begin{split}
        \left \{
        \begin{array}{cc}
           \text{it is $\mr{G}_r(er-(e-1)k,e)$-free,
 and is $\mr{G}_r(ir-(i-1)k-1,i)$-free for all $t\le i\le e-1$;}  &  \\
            \text{for every $(k-1)$-subset $T$, the number of edges containing $T$ is either $0$ or at least $e$.} &
        \end{array}
        \right.
    \end{split}
\end{align}

\noindent Then it is clear that
\begin{align*}
    f_r(n,er-(e-1)k,e)\ge f_r^{(e-1)}(n,er-(e-1)k,e)\ge\cdots\ge f_r^{(2)}(n,er-(e-1)k,e)
\end{align*}
and hence
\begin{equation}\label{eq:onedirection}
\begin{aligned}
    \limsup_{n\rightarrow\infty}\frac{f_r(n,er-(e-1)k,e)}{n^k}&\ge\limsup_{n\rightarrow\infty}\frac{f_r^{(e-1)}(n,er-(e-1)k,e)}{n^k}\\
    &\ge\cdots\ge\limsup_{n\rightarrow\infty}\frac{f_r^{(2)}(n,er-(e-1)k,e)}{n^k}.
\end{aligned}
\end{equation}
Moreover, as an easy consequence of \cref{lem:exist-0} we have the following inequality:

\begin{equation}\label{eq:exist}
\begin{aligned}
    \liminf_{n\rightarrow\infty}\frac{f_r(n,er-(e-1)k,e)}{n^k}\ge \limsup_{n\rightarrow\infty}\frac{f_r^{(2)}(n,er-(e-1)k,e)}{n^k}.
\end{aligned}
\end{equation}
Therefore, to prove the existence of $\lim_{n\rightarrow\infty}\frac{f_r(n,er-(e-1)k,e)}{n^k}$, it suffices to show
\begin{equation}\label{eq:recursive}
\begin{aligned}
    \limsup_{n\rightarrow\infty}\frac{f_r(n,er-(e-1)k,e)}{n^k}&\le\limsup_{n\rightarrow\infty}\frac{f_r^{(e-1)}(n,er-(e-1)k,e)}{n^k}\\&\le\cdots\le\limsup_{n\rightarrow\infty}\frac{f_r^{(2)}(n,er-(e-1)k,e)}{n^k}.
\end{aligned}
\end{equation}

\noindent Indeed, it follows by \eqref{eq:onedirection} and \eqref{eq:recursive} that $\limsup_{n\rightarrow\infty}\frac{f_r(n,er-(e-1)k,e)}{n^k}=\limsup_{n\rightarrow\infty}\frac{f_r^{(2)}(n,er-(e-1)k,e)}{n^k}$, which together with \eqref{eq:exist} would imply that $\limsup_{n\rightarrow\infty}\frac{f_r(n,er-(e-1)k,e)}{n^k}=\liminf_{n\rightarrow\infty}\frac{f_r(n,er-(e-1)k,e)}{n^k}$, and hence the existence of the corresponding limit. However, we are only able to prove \eqref{eq:recursive} for $k=2$, stated as below.

\begin{lemma}\label{lem:recursive}
For all fixed $r\ge 3$ and $e\ge 3$,
\begin{align*}
    \limsup_{n\rightarrow\infty}\frac{f_r(n,er-2(e-1),e)}{n^2}&\le\limsup_{n\rightarrow\infty}\frac{f_r^{(e-1)}(n,er-2(e-1),e)}{n^2}\\&\le\cdots\le\limsup_{n\rightarrow\infty}\frac{f_r^{(2)}(n,er-2(e-1),e)}{n^2}.
\end{align*}
\end{lemma}

In Delcourt and Postle's solution to \cref{con:besconjecture}, their proof was based on a result which is similar to \cref{lem:recursive} for the special case $r=3$. Using our notation, they essentially proved that $\limsup_{n\rightarrow\infty}\frac{f_3(n,e+2,e)}{n^2}\le\limsup_{n\rightarrow\infty}\frac{f^{(2)}_3(n,e+2,e)}{n^2}$ (see Theorem 1.7 in \cite{bes}). To prove that result, they first proved a structural lemma for $\ma{G}_3(e+2,e)$-free 3-graphs which are not $\ma{G}_3(i+1,i)$-free for some $2\le i\le e+1$ (see Lemma 2.1 in \cite{bes}). Our proof of \cref{lem:recursive} is inspired by their approach. At the core of our proof is a more powerful structural lemma (see \cref{lem:structure} below) for $r$-graphs that satisfy the assumptions of \eqref{eq:property} but are not $\mr{G}_r((t-1)r-(t-2)k-1,t-1)$-free. Based on this new structural lemma and a density increment argument, we are able to show that for $k=2$, all such $r$-graphs must contain a sufficiently large $\mr{G}_r((t-1)r-(t-2)k-1,t-1)$-free subhypergraph with non-decreasing edge density (see \cref{lem:crucial} below).  This enables us to prove the inequality $\limsup_{n\rightarrow\infty}\frac{f_r^{(t)}(n,er-2(e-1),e)}{n^2}\le\limsup_{n\rightarrow\infty}\frac{f_r^{(t-1)}(n,er-2(e-1),e)}{n^2}$ for every fixed $3\le t\le e-1$. We suspect that the above assertion holds for all $k\ge 3$. If so, it would settle \cref{que:ST-1} in the full generality.

The rest of this paper is organized as follows. The proof of \cref{lem:recursive} is presented in \cref{sec:recursive}. In \cref{sub:first} we will prove the first inequality of \cref{lem:recursive}. In \cref{sub:struc} we will state and prove our structural and density lemmas (see \cref{lem:structure} and \cref{lem:crucial} below). In \cref{sub:remaining} we will prove the remaining inequalities of \cref{lem:recursive}. 


\section{Proof of \cref{lem:recursive}}\label{sec:recursive}
\subsection{The first inequality}\label{sub:first}
In this subsection, we will prove the first inequality of \cref{lem:recursive}, and in fact we can prove it for every $k\ge 2$. Let us begin with some notations. For a positive integer $n$, let $[n]=\{1,\ldots,n\}$. For a finite set $X$ and a positive integer $r$, let $\binom{X}{r}$ denote the collection of all distinct $r$-element subsets of $X$. Therefore, $\binom{[n]}{r}=\{A\subseteq [n]:|A|=r\}$. For an $r$-graph $\ma{F}\subseteq\binom{[n]}{r}$ and a subset $T\subseteq[n]$, the {\it codegree} of $T$ in $\ma{F}$, denoted by $\deg_{\ma{F}}(T)$, is the number of edges in $\ma{H}$ which contain $T$ as a subset, i.e., $\deg_{\ma{F}}(T)=|\{A\in\ma{F}:T\subseteq A\}|.$

The following lemma shows that any $r$-graph can be made to have no $(k-1)$-subset of codegree belonging to $\{1,\ldots,e-1\}$ by deleting at most $(e-1)\binom{n}{k-1}$ of its edges.

\begin{lemma}\label{lem:easy}
  For every $r$-graph $\ma{F}\subseteq\binom{[n]}{r}$ there exists a subhypergraph $\ma{F}'\subseteq\ma{F}$ such that $|\ma{F}'|\ge\max\{0,|\ma{F}|-(e-1)\binom{n}{k-1}\}$ and for every $T\in\binom{[n]}{k-1}$, we have either $\deg_{\ma{F}'}(T)=0$ or $\deg_{\ma{F}'}(T)\ge e$.
\end{lemma}

\begin{proof}
  Successively remove the edges of $\ma{F}$ which contain at least one $(k-1)$-subset of codegree at least one and at most $e-1$, where the codegree is counted with respect to the yet unmoved edges of $\ma{F}$. Let $A_i$ be the $i$-th removed edge, and $T_i$ be some $(k-1)$-subset of codegree at least one and at most $e-1$ contained in $A_i$. In such a case we say that $A_i$ is removed due to $T_i$ and $T_i$ is responsible for $A_i$. On one hand, every $(k-1)$-subset of $[n]$ can be responsible for at most $e-1$ edges, since otherwise it would have codegree at least $e$ and could not be responsible for any removed edge in the beginning. On the other hand, every removed edge is removed due to some edge. Therefore we conclude that the removal process must terminate after at most $(e-1)\binom{n}{k-1}$ edge removals. Note that the resulting $r$-graph is possibly empty.
\end{proof}

The next lemma proves a lower bound of $f_r^{(e-1)}(n,er-(e-1)k,e)$ in terms of $f_r(n,er-(e-1)k,e)$.

\begin{lemma}\label{lem:case1}
    $f_r(n,er-(e-1)k,e)-(e-1)\binom{n}{k-1}\le f_r^{(e-1)}(n,er-(e-1)k,e)).$
\end{lemma}

\begin{proof}
Let $\ma{F}$ be an arbitrary $n$-vertex $\mr{G}_r(er-(e-1)k,e)$-free $r$-graph. Let $\ma{F}'$ be the subhypergraph of $\ma{F}$ given by \cref{lem:easy}. To prove the lemma, it suffices to show that $\ma{F}'$ satisfies the condition of \eqref{eq:property} for $t=e-1$. Therefore, it is enough to show that $\ma{F}'$ is $\mr{G}_r((e-1)r-(e-2)k-1,e-1)$-free.

Assume otherwise, then there exist distinct edges $A_1,\ldots,A_{e-1}\in\ma{F}'$ such that $|\cup_{i=1}^{e-1} A_i|\le (e-1)r-(e-2)k-1$. Let $T$ be an $(k-1)$-subset of $A_1$. According to \cref{lem:easy}, there exists $A_e\in\ma{F}'\setminus\{A_1,\ldots,A_{e-1}\}$ such that $T\subseteq A_e$. It follows that
    \begin{align*}
        |\cup_{i=1}^e A_i|\le |\cup_{i=1}^{e-1} A_i|+|A_e\setminus A_1|\le \big((e-1)r-(e-2)k-1\big)+\big(r-k+1\big)=er-(e-1)k,
    \end{align*}
    contradicting the fact that $\ma{F}'$ is $\mr{G}_r(er-(e-1)k,e)$-free.
\end{proof}

\begin{proof} [Proof of the first inequality of \cref{lem:recursive}]
According to the lemma above, it is easy to see that
\begin{align}\label{eq:ineq-1}
    \limsup_{n\rightarrow\infty}\frac{f_r(n,er-(e-1)k,e)}{n^k}\le\limsup_{n\rightarrow\infty}\frac{f_r^{(e-1)}(n,er-(e-1)k,e)}{n^k}.
\end{align}
Setting $k=2$ in \eqref{eq:ineq-1} proves the first inequality of \cref{lem:recursive}. 
\end{proof}

\subsection{The structural and density lemmas}\label{sub:struc}

First of all let us mention the following easy lemma.

\begin{lemma}\label{lem:upp}
Let $n$ be sufficiently large enough as a function of $e$ and $r$. Then
\begin{itemize}
    \item [{\rm (i)}] $f_r(n,er-(e-1),e)<\frac{n}{r-1}$;
    \item [{\rm (ii)}] if $\ma{F}\subseteq\binom{[n]}{r}$ is $\mr{G}_r(er-(e-1)k,e)$-free, then for every $T\in\binom{[n]}{k-1}$, $\deg_{\ma{F}}(T)<\frac{n}{r-k}$. 
\end{itemize}
\end{lemma}

\begin{proof}
   The first inequality is an easy consequence of \eqref{eq:k=1}. To prove the second inequality, it is straightforward to  verify by definition that for every $T\in\binom{n}{k-1}$, the $(r-k+1)$-graph $\{A\setminus T:A\in\ma{F}\}$ defined on the vertex set $[n]\setminus T$ is $\mr{G}_{r-k+1}(e(r-k+1)-(e-1),e)$-free, and therefore the result follows by (i).
\end{proof}

Next we present a structural lemma for $\mr{G}_r(er-(e-1)k,e)$-free $r$-graphs which are not $\mr{G}_r(ir-(i-1)k-1,i)$-free for some $2\le i\le e-1$.

\begin{lemma}\label{lem:structure}
    Let $t\ge 3$ and $\ma{F}\subseteq\binom{[n]}{r}$ be an $r$-graph satisfying the assumptions of \eqref{eq:property}.
    If it is not $\mr{G}_r((t-1)r-(t-2)k-1,t-1)$-free, i.e., there are $t-1$ distinct edges $A_1,\ldots,A_{t-1}\in\ma{F}$ such that $|\cup_{i=1}^{t-1} A_i|\le (t-1)r-(t-2)k-1$, then the following hold:
    \begin{itemize}
        \item [{\rm (i)}] let $X= \cup_{i=1}^{t-1} A_i$, then $|X|=(t-1)r-(t-2)k-1$;
        \item [{\rm (ii)}] for every $A\in\ma{F}\setminus\{A_1,\ldots,A_{t-1}\}$, $|A\cap X|\le k-1$;
        \item [{\rm (iii)}] let $\ma{I}(X)=\big\{A\in\ma{F}\setminus\{A_1,\ldots,A_{t-1}\}:|A\cap X|= k-1\big\}$, then
        $$|\ma{I}(X)|\le \max\big\{e-t+1,f_{r-k+1}(n-|X|,(e-t+1)(r-k+1)-(e-t),e-t+1)\big\}.$$ In particular, if $n$ is large enough as a function of $r$ and $e$, then $|\ma{I}(X)|<\frac{n-|X|}{r-k}$.
    \end{itemize}
\end{lemma}

\begin{proof}
    (i) For the ease of notation let $x=(t-1)r-(t-2)k-1$. Assume for contradiction that $|X|\le x-1$. Let $T$ be an $(k-1)$-subset of $A_1$. As $\deg_{\ma{F}}(T)\ge e$ there exists $A_t\in\ma{F}\setminus\{A_1,\ldots,A_{t-1}\}$ such that $T\subseteq A_t$. It follows that
    \begin{align*}
        |\cup_{i=1}^t A_i|\le |X|+|A_t\setminus A_1|\le (x-1)+(r-k+1)=tr-(t-1)k-1,
    \end{align*}
   violating the assumption that $\ma{F}$ is $\mr{G}_r(tr-(t-1)k-1,t)$-free.

    (ii) If there exists $A_t\in\ma{F}\setminus\{A_1,\ldots,A_{t-1}\}$ such that $|A_t\cap X|\ge k$, then similar to the proof of (i), it is clear that
    $$|\cup_{i=1}^{t} A_i|\le |X|+|A_t\setminus X|\le x+(r-k)=tr-(t-1)k-1,$$ a contradiction.

    (iii) Assume that $|\ma{I}(X)|\ge e-t+1$, otherwise there is nothing to prove. As $\ma{F}$ is $\mr{G}_r(er-(e-1)k,e)$-free, for arbitrary $e-t+1$ distinct edges $A_t,\ldots,A_{e}\in\ma{I}(X)$ we have
    \begin{align*}
        |\cup_{i=1}^e A_i|=|X|+|\cup_{i=t}^e (A_i\setminus X)|\ge er-(e-1)k+1,
    \end{align*}
    which implies that
    \begin{align*}
        |\cup_{i=t}^e (A_i\setminus X)|\ge er-(e-1)k+1-|X|=(e-t+1)(r-k+1)-(e-t)+1.
    \end{align*}
    It follows that the $(r-k+1)$-graph
    $\ma{J}(X):=\{A\setminus X:A\in\ma{I}(X)\}$
    defined on the vertex set $V(\ma{F})\setminus X$ is $\mr{G}_{r-k+1}((e-t+1)(r-k+1)-(e-t),e-t+1)$-free.
    Moreover, it is clear that for all distinct $A,B\in\ma{I}(X)$, we have $A\setminus X\neq B\setminus X$, since otherwise
   $$|\cup_{i=1}^{t-1} A_i\cup A\cup B|\le x+r-k+1=tr-(t-1)k,$$ contradicting the fact that $\ma{F}$ is $\mr{G}_r((t+1)r-tk-1,t+1)$-free.
    We conclude that 
    \begin{align*}
        |\ma{I}(X)|=|\ma{J}(X)|\le f_{r-k+1}(n-|X|,(e-t+1)(r-k+1)-(e-t),e-t+1),
    \end{align*}
    as needed. Finally, the last sentence of (iii) follows by \cref{lem:upp} (i).
\end{proof}

Applying \cref{lem:structure} with $k=2$ yields the following result, which shows that any $r$-graph that satisfy the assumptions of \eqref{eq:property} but are not $\mr{G}_r((t-1)r-2(t-2)-1,t-1)$-free must contain a sufficiently large $\mr{G}_r((t-1)r-2(t-2)-1,t-1)$-free subhypergraph with non-decreasing edge density. Note that we did not try to optimize the constants written in \cref{lem:crucial}.

\begin{lemma}\label{lem:crucial}
    Let $t\ge 3$ and $r\ge 3$. Let $n$ be sufficiently large as a function of $e$ and $r$. Let $\ma{F}\subseteq\binom{[n]}{r}$ be an $r$-graph satisfying the assumptions of \eqref{eq:property}. Moreover, assume that $\frac{|\ma{F}|}{n^2}\ge\max\{\frac{1}{r(r-1)}-\frac{\delta}{r(r-2)},\frac{1}{2r^2-2r-5}\}$, where $\delta=\frac{1}{(3r-4)(r-1)+1}$. Then there exists a subhypergraph $\ma{F}'\subseteq\ma{F}$ with the following properties:
    \begin{itemize}
        \item [{\rm (i)}] $\ma{F}'$ is $\mr{G}_r((t-1)r-2(t-2)-1,t-1)$-free;
        \item [{\rm (ii)}] $|V(\ma{F}')|\ge \alpha n$, where $\alpha=\sqrt{\frac{3r^2-4r-14}{3r^2+2r-8}}$;
        \item [{\rm (iii)}] $\frac{|\ma{F}'|}{|V(\ma{F}')|^2}\ge \frac{|\ma{F}|}{n^2}$.
    \end{itemize}
\end{lemma}

\begin{proof}
  Let $x=(t-1)r-2(t-2)-1$. Assume that $\ma{F}$ is not $\mr{G}_r(x,t-1)$-free, otherwise there is nothing to prove. Call a family of $t-1$ distinct edges $\{A_1,\ldots,A_{t-1}\}\subseteq\ma{F}$ a bad configuration if $|\cup_{i=1}^{t-1} A_i|\le x$. By \cref{lem:structure} (i), the union of every bad configuration has size exactly $x$.

  Let $\ma{F}_0=\ma{F}$. We will successfully construct hypergraphs $\ma{F}_1,\ma{F}_2,\ldots$ as follows. We aim to show that such a process will terminate and produce a hypergraph as needed. For $j\ge 0$, if $\ma{F}_j$ is not $\mr{G}_r(x,t-1)$-free, say, it contains a bad configuration $\{A_{1,j},\ldots,A_{t-1,j}\}$ with $X_j:=\cup_{i=1}^{t-1} A_{i,j}$. Then we set $\ma{F}_{j+1}$ to be the subhypergraph of $\ma{F}_j$ formed by
  \begin{align*}
      \ma{F}_{j+1}:=\ma{F}_j\setminus(\ma{I}(X_j)\cup\{A_{1,j},\ldots,A_{t-1,j}\}) \quad\text{and}\quad V(\ma{F}_{j+1}):=V(\ma{F}_j)\setminus X_j
  \end{align*}
  where $\ma{I}(X_j)$ is the $r$-graph defined in \cref{lem:structure} with $k=2$ and $A_{1,j},\ldots,A_{t-1,j},X_j$ playing the roles of $A_1,\ldots,A_{t-1},X$, respectively.

  Before going any further, we need to show that $V(\ma{F}_{j+1})$ is well-defined, namely, for every $A\in\ma{F}_{j+1}$, we have $A\cap X_j=\emptyset$. Indeed, by applying \cref{lem:structure} (ii) with $k=2$ and the definition of $\ma{I}(X_j)$, it is not hard to see that
\begin{align*}
    \{A\in\ma{F}_j:A\cap X_j\neq\emptyset\}=\{A\in\ma{F}_j:|A\cap X_j|\ge 1\}=\ma{I}(X_j)\cup\{A_{1,j},\ldots,A_{t-1,j}\},
\end{align*}
which proves that $V(\ma{F}_{j+1})$ is well-defined.

For each $j\ge 0$, denote $e_j=|\ma{F}_j|$ and $v_j=|V(\ma{F}_j)|$.
We claim that for $0\le j\le \frac{(1-\alpha) n}{x}$, the following hold:
  \begin{align}\label{eq:e_j}
      e_{j+1}> e_{j}-\frac{v_j}{r-2} \quad\text{and}\quad  v_{j+1}=v_j-x.
  \end{align}
Indeed, by definition we have $e_{j+1}=e_j-|\ma{I}(X_j)|-(t-1)$; moreover, by applying \cref{lem:structure} (iii) with $k=2$ we have for large enough $n$ (thus $v_j$ is also large enough for $0\le j\le \frac{(1-\alpha) n}{x}$),
\begin{align*}
    |\ma{I}(X_j)|+(t-1)<\frac{v_j-x}{r-2}+(t-1)<\frac{v_j}{r-2},
\end{align*}
which proves the inequality. The equality follows from the definition of $V(\ma{F}_{j+1})$ and the fact that $X_j\subseteq V(\ma{F}_{j})$.

Applying \eqref{eq:e_j} repeatedly yields that for every $0\le j\le \frac{(1-\alpha) n}{x}$,
\begin{align}\label{eq:j}
    e_j>e_0-\frac{1}{r-2}\sum_{i=0}^{j-1}v_i\quad\text{and}\quad v_j=v_0-xj.
\end{align}
Since $\ma{F}_j\subseteq\ma{F}$, it has to be $\mr{G}_r(er-2(e-1),e)$-free. Therefore, it follows by \cref{lem:upp} (ii) that the average degree of $V(\ma{F}_j)$ must be less than $\frac{v_j}{r-2}$, or equivalently,
\begin{align}\label{eq:degree}
    \frac{e_jr}{v_j}<\frac{v_j}{r-2}.
\end{align}
According to \eqref{eq:j} and \eqref{eq:degree}, we have
\begin{align}\label{eq:ratio}
    \frac{e_0-\frac{v_0j}{r-2}+\frac{j(j-1)x}{2(r-2)}}{(v_0-xj)^2}<\frac{e_j}{v_j^2}<\frac{1}{r(r-2)}.
\end{align}
Substituting $v_0=n,e_0\ge(\frac{1}{r(r-1)}-\frac{\delta}{r(r-2)})n^2$, setting $y=xj$, and rearranging yields that
\begin{align}\label{eq:y}
    y(2n-y)-\frac{ry}{2x}(2n-y)-\frac{ry}{2}<(\frac{1}{r-1}+\delta)n^2.
\end{align}
Note that $\frac{ry}{2}\le\frac{rn}{2}$. Observe that $x$ is at least the size of the union of two distinct $r$-subsets, so $x\ge r+1$. Substituting the above inequalities to \eqref{eq:y} gives that
\begin{align*}
    \frac{r+2}{2r+2}n^2-\frac{r+2}{2r+2}(n-y)^2<(\frac{1}{r-1}+\delta)n^2+\frac{rn}{2}<\frac{3n^2}{3r-4},
\end{align*}
where the last inequality follows from the choice of $\delta$ and the assumption that $n$ is sufficiently large as a function of $r$. Rearranging gives that
\begin{align*}
    (n-y)^2>\alpha^2 n^2.
\end{align*}
As $y\le n$, it follows that $y<(1-\alpha)n$ and hence $j<\frac{(1-\alpha)n}{x}$. We conclude that the above process must stop in at most $\frac{(1-\alpha)n}{x}$ steps, and therefore at most $(1-\alpha)n$ vertices of $V(\ma{F})$ could be removed. By our construction, it is clear that the resulting $r$-graph, denoted by $\ma{F}'$, satisfies conditions (i) and (ii) in the lemma. 

In order to prove that $\ma{F}'$ also satisfies (iii), it is enough to show that for every $j\ge 1$, $\frac{e_j}{v_j^2}>\frac{e_0}{v_0^2}$. Indeed, by \eqref{eq:ratio} we have
\begin{align*}
    \frac{e_j}{v_j^2}-\frac{e_0}{v_0^2}>\frac{e_0-\frac{v_0j}{r-2}+\frac{j(j-1)x}{2(r-2)}}{(v_0-xj)^2}-\frac{e_0}{v_0^2}.
\end{align*}
Plugging $y=jx$ and rearranging, we have
\begin{align*}
    \frac{e_j}{v_j^2}-\frac{e_0}{v_0^2}>\frac{1}{(n-y)^2}\Big(\frac{e_0}{v_0^2}(2ny-y^2)-\frac{ny}{(r-2)x}+\frac{y^2}{2(r-2)x}-\frac{y}{2(r-2)}\Big)
\end{align*}
Observe that $-\frac{ny}{(r-2)x}+\frac{y^2}{2(r-2)x}$ is negative. Then as $x\ge r+1$ and $\frac{e_0}{v_0^2}\ge\frac{1}{2r^2-2r-5}$, we have
\begin{align*}
    \frac{e_j}{v_j^2}-\frac{e_0}{v_0^2}>\frac{y}{(n-y)^2}\Big(\frac{n}{(2r^2-2r-5)(r^2-r-2)}-\frac{y/2}{(2r^2-2r-5)(r^2-r-2)}-\frac{1}{2(r-2)}\Big)>0,
\end{align*}
completing the proof of (iii), where the last inequality follows from the fact that $y/2<n/2$ and the assumption that $n$ is sufficiently large.
\end{proof}

\subsection{The remaining inequalities}\label{sub:remaining}
In this subsection we will prove the remaining inequalities of \cref{lem:recursive}. We will make use of the following result, proved independently in \cite{Delcourtconflict,Glockconflict}.

\begin{lemma}[see Theorem 1.3 in \cite{Delcourtconflict} and Theorem 1.1 in \cite{Glockconflict}]\label{lem:packing}
Let $r\ge 3$ and $e\ge 3$ be fixed integers. Then there exists $n_0$ and $\epsilon>0$ such that for all $n\ge n_0$ there exists an $n$-vertex $r$-graph which has at least $(1-n^{-\epsilon})\frac{n^2}{r^2-r}$ edges and is $\mr{G}_r(tr-2(t-1),t)$-free for every $2\le t\le e$.
\end{lemma}

We will use the following easy consequence of \cref{lem:packing}.

\begin{lemma}\label{lem:packing-modified}
Let $r\ge 3$ and $e\ge 3$ be fixed integers. Then there exists $n_0$ such that for all $n\ge n_0$ there exists an $n$-vertex $r$-graph $\ma{F}$ which satisfies $\frac{|\ma{F}|}{n^2}>\max\{\frac{1}{r(r-1)}-\frac{\delta}{r(r-2)},\frac{1}{2r^2-2r-5}\}$, where $\delta=\frac{1}{(3r-4)(r-1)+1}$, and the assumptions of \eqref{eq:property} for $k=2$.
\end{lemma}

\begin{proof}
    It is clear that the $r$-graph given by \cref{lem:packing} satisfies the first line of \eqref{eq:property}. It is also clear by \cref{lem:easy} that the $r$-graph given by \cref{lem:packing} can be made to satisfy the second line of \eqref{eq:property} by removing at most $(e-1)n$ of its edges. Lastly, the required lower bound on $\frac{|\ma{F}|}{n^2}$ holds for sufficiently large $n$ since for every $r\ge 3$, $\frac{1}{r^2-r}$ is strictly larger than both $\frac{1}{r(r-1)}-\frac{\delta}{r(r-2)}$ and $\frac{1}{2r^2-2r-5}$.
\end{proof}

Given Lemmas \ref{lem:crucial} and \ref{lem:packing-modified}, the remaining inequalities of \cref{lem:recursive} can be proved by standard analysis. We will follow the proof of Theorem 1.7 in \cite{bes}.

\begin{proof}[Proof of the remaining inequalities of \cref{lem:recursive}] For every $3\le t\le e-1$, let $$a_t=\limsup_{n\rightarrow\infty}\frac{f_r^{(t)}(n,er-2(e-1),e)}{n^2}.$$ It suffices to show that for every $t$, $a_t\le a_{t-1}$.

Assume for the sake of contradiction that there exists some $3\le t\le e-1$ such that $a_t>a_{t-1}$. Let $\epsilon=a_t-a_{t-1}$, $b=\max\{\frac{1}{r(r-1)}-\frac{\delta}{r(r-2)},\frac{1}{2r^2-2r-5}\}$ with $\delta=\frac{1}{(3r-4)(r-1)+1}$, and $c=\max\{a_t-\epsilon/2,b\}$. Since $a_t>b$ by \cref{lem:packing-modified}, it follows that
\begin{align*}
 a_t>c\ge a_t-\epsilon/2>a_{t-1}.
\end{align*}
Let
\begin{align*}
    \ma{I}_t=\{n\in\mathbb{N}:\frac{f_r^{(t)}(n,er-2(e-1),e)}{n^2}\ge c\}
\end{align*}
and
\begin{align*}
    \ma{I}_{t-1}=\{n\in\mathbb{N}:\frac{f_r^{(t-1)}(n,er-2(e-1),e)}{n^2}\ge c\}.
\end{align*}

Then, as $a_t>c$, $\ma{I}_t$ is infinite. As $a_{t-1}<c$, $\ma{I}_{t-1}$ is finite. Let $n_{t-1}=\max\{n:n\in\ma{I}_{t-1}\}$.
Since $\ma{I}_t$ is infinite, there exists $n_t\in\ma{I}_t$ such that $n_t>n_{t-1}/\alpha$, where $\alpha$ is defined in \cref{lem:crucial} (ii). According to the definition of $\ma{I}_t$, there exists an $n_t$-vertex $r$-graph $\ma{F}$ which satisfies the conditions of \eqref{eq:property} with $k=2$, and $\frac{|\ma{F}|}{|V(\ma{F})|^2}\ge c$. As $\ma{F}$ also satisfies the assumptions of \cref{lem:crucial}, it follows that there exists a subhypergraph $\ma{F}'\subseteq\ma{F}$ that satisfies all of the three properties listed in \cref{lem:crucial}. In particular, we have $\frac{|\ma{F}'|}{|V(\ma{F}')|^2}\ge \frac{|\ma{F}|}{n^2}\ge c$, which implies that $|V(\ma{F}')|\in\ma{I}_{t-1}$. However, we also have $|V(\ma{F}')|\ge\alpha |V(\ma{F})|>n_{t-1}$, which contradicts the maximality of $n_{t-1}$.
\end{proof}



\section*{Acknowledgements}
The research  of Chong Shangguan is supported by the National Key Research and Development Program of China under Grant No. 2021YFA1001000, the National Natural Science Foundation of China under Grant No. 12101364, and the Natural Science Foundation of Shandong Province under Grant No. ZR2021QA005.

\bibliographystyle{plain}
\bibliography{degenerate_turan_2}
 \end{document}